\documentclass[12pt]{article}
\usepackage{graphicx,amssymb,amsmath,amsthm}

\title{Enumeration of $(0,1)$-matrices avoiding \\some $2\times2$ matrices}

\author{Hyeong-Kwan Ju\thanks{This study was financially supported by Chonnam National University, 2008-2329.}\\
\small Department of Mathematics\\[-0.8ex]
\small Chonnam National University, \\[-0.8ex]
\small Kwangju 500-757, South Korea\\
\small \texttt{hkju@chonnam.ac.kr}\\
\and
Seunghyun Seo\thanks{This study was supported by 2008 Research Grant from Kangwon National University.}\\
\small Department of Mathematics Education\\[-0.8ex]
\small Kangwon National University\\[-0.8ex]
\small Chuncheon 200-701, South Korea\\
\small \texttt{shyunseo@kangwon.ac.kr}
}

\date{\today }

\newcommand{\by}{\times}
\newcommand{\kn}{k \times n}
\newcommand{\mS}{\mathcal S}
\newcommand{\MTad}{\begin{pmatrix}1 & 0 \\ 0 & 1 \\ \end{pmatrix}}

\newcommand{\MTab}{\begin{pmatrix}1 & 1 \\ 0 & 0 \\ \end{pmatrix}}
\newcommand{\sMTab}{\left(\begin{smallmatrix}1 & 1 \\ 0 & 0 \\ \end{smallmatrix}\right)}

\newcommand{\sMTcd}{\left(\begin{smallmatrix}0 & 0 \\ 1 & 1 \\ \end{smallmatrix}\right)}

\newcommand{\MTac}{\begin{pmatrix}1 & 0 \\ 1 & 0 \\ \end{pmatrix}}

\newcommand{\MTOd}{\begin{pmatrix}1 & 1 \\ 1 & 0 \\ \end{pmatrix}}

\newcommand{\MTIa}{\begin{pmatrix}1 & 0 \\ 0 & 0 \\ \end{pmatrix}}

\newcommand{\MTI}{\begin{pmatrix}1 & 1 \\ 1 & 1 \\ \end{pmatrix}}
\newcommand{\sMTI}{\left(\begin{smallmatrix}1 & 1 \\ 1 & 1 \\ \end{smallmatrix}\right)}
\newcommand{\MTO}{\begin{pmatrix}0 & 0 \\ 0 & 0 \\ \end{pmatrix}}

\newtheorem{thm}{Theorem}[section]

\newtheorem{dfn}[thm]{Definition}
\newtheorem{prop}[thm]{Proposition}
\theoremstyle{remark}
\newtheorem{rmk}{Remark}
\newtheorem{eg}{Example}

\begin{document}
\maketitle
\begin{abstract}
We enumerate the number of $(0,1)$-matrices avoiding $2\times 2$ submatrices satisfying certain conditions.
We also provide corresponding exponential generating functions.
\end{abstract}

\section{Introduction}

Let $M(k,n)$ be the set of $k \by n$ matrices with entries $0$ and $1$. It is obvious that the number of elements in the
set $M(k,n)$ is $2^{kn}$. It would be interesting to consider the number of elements in $M(k,n)$ with certain conditions.
For example, how many matrices of $M(k,n)$ do not have $2\by 2$ submatrices of the forms $\sMTab$ and $\sMTcd$? In this article we will
give answers to the previous question and other questions.

Consider $M(2,2)$, the set of all possible $2\by 2$ submatrices.
For two elements $P$ and $Q$ in $M(2,2)$, we denote $P\sim Q$ if $Q$ can be obtained from $P$ by row or column exchanges. It is obvious that $\sim$ is an equivalence relation on $M(2,2)$. With this equivalence relation, $M(2,2)$ is partitioned with seven equivalent classes having the following seven representatives. %\begin{eqnarray*}
%I &:=& \MTad\\
%\Gamma &:=& \MTOd \\
%C &:=& \MTIa\\
%L &:=& \MTab\\
%T &:=& \MTac\\
%J &:=& \MTI \\
%O &:=& \MTO
%\end{eqnarray*}
\begin{gather*}
I = \MTad,\quad
\Gamma = \MTOd ,\quad
C = \MTIa,\quad \\
T = \MTab,\quad
L = \MTac,\quad
J = \MTI,\quad
O = \MTO.
\end{gather*}
Here $C$, $T$, and $L$ mean ``corner", ``top", and ``left", respectively. Let $\mS$ be the set of
these representatives, i.e.,
$$\mS:=\{I, \,\Gamma,\, C,\, T,\, L,\, J, \,O\}.$$

Given an element $S$ in $\mS$, a matrix $A$ is an element of the set $M(S)$ if and only if for {\em every} permutation $\pi_1$ of the rows and $\pi_2$ of the columns, the resulting matrix does not have the submatrix $S$.
Equivalently, $A \in M(S)$ means that $A$ has no submatrices in the equivalent class $[S]$.
For a subset $\alpha$ of $\mS$, $M(\alpha)$ is defined by the set
$\cap_{S \in \alpha} M(S)$.
Note that the definition of $M(S)$ (also $M(\alpha)$) is different from that in \cite{KRW, Sp2}.
If $A$ belongs to $M(\alpha)$, then we say that $A$ {\em avoids} $\alpha$.
We let $\phi(k,n;\alpha)$ be the number of $k\by n$ $(0,1)$-matrices
in $M(\alpha)$.
%which do not have $2\by 2$ submatrices in $\alpha$, where $\alpha$ is a subset of $M(2,2)$.

Our goal is to express $\phi(k,n;\alpha)$ in terms of $k$ and $n$ explicitly for each subset $\alpha$ of the set $\mS$. For $|\alpha|=1$,
We can easily notice that $\phi(k,n;\Gamma )=\phi(k,n;C )$ and  $\phi(k,n;J )=\phi(k,n;O )$ by swapping $0$ and $1$. We also notice that $\phi(k,n;T )=\phi(n,k;L )$ by transposing the matrices.
The number $\phi(k,n;I)$ is well known (see \cite{Br, Ka, KKL}) and $(0,1)$-matrices
avoiding type $I$ are called $(0,1)$-lonesum matrices (we will define and discuss this in~\ref{sub:ls}).
In fact, lonesum matrices are the primary motivation of this article and its corresponding work.
The study of $M(J)$ (equivalently $M(O)$) appeared in \cite{jrKim, KRW, Sp2}, but finding a closed form of $\phi(k,n;J)=\phi(k,n;O)$ is still open.
The notion of ``$\Gamma$-free matrix" was introduced by Spinrad \cite{Sp1}. He dealt with a totally balanced matrix which has a permutation of the rows and coluums that are $\Gamma$-free. We remark that the set of totally balanced matrices is different from $M(\Gamma)$. %Except for the aforementioned case, we have not found any literature concerning other cases.
%For $|\alpha|=2$,
%Kitaev et el.~\cite[Section 3]{KMV} studied this subject as a generalization of pattern avoidance in words. In particular they enumerated the set $M(\{J,O\})$. For other relevant papers see~\cite{Ki,KR}.

%Moreover, given the equivalence relation $\sim$, if we define the new equivalent relation $P\sim ' Q$ by $P\sim Q$ or $P=Q^{\rm t}$, then $T \cup L$ becomes a single equivalent class.
%Also if we define another new equivalent relation $P\sim '' Q$ by $P\sim Q$ or $P=\sMTI - Q$, then $\Gamma \cup C$ and $J \cup O$ become a single equivalent class respectively.

In this paper we calculate $\phi(k,n;\alpha)$, where $\alpha$'s are $\{\Gamma\}$ (equivalently $\{C\}$) and $\{T\}$ (equivalently $\{L\}$). We also enumerate $M(\alpha)$ where $\alpha$'s are $\{\Gamma ,C\}$, $\{ T,L\}$, and $\{ J,O\}$. For the other subsets of $\mS$, we discuss them briefly in the last section. Note that some of our result (subsection~\ref{subJO}) is an independent derivation of some of
the results in \cite[section 3]{KMV} by Kitaev et al.; for other relevant papers see~\cite{Ki, KR}.
%Finding a closed form of $\phi(k,n;J)=\phi(k,n;O)$ is still open to us.

\section{Preliminaries}
\subsection{Definitions and Notations} \label{sub:dn}
A matrix $P$ is called $(0,1)$-matrix if all the entries of $P$ are $0$ or $1$. From now on we will consider $(0,1)$-matrices only, so we will omit ``$(0,1)$" if it causes no confusion.  Let $M(k,n)$ be the set of $\kn$-matrices. Clearly, if $k, n \geq 1$, $M(k,n)$ has $2^{kn}$ elements. For convention we assume that $M(0,0)=\{\emptyset \}$ and $M(k,0)=M(0,n)=\emptyset$ for positive integers $k$ and $n$.

Given a matrix $P$, a submatrix of $P$ is formed by selecting certain rows and columns from $P$. For example, if $P=\left(\begin{smallmatrix} a&b&c&d \\ e&f&g&h \\ i&j&k&l \end{smallmatrix}\right) $, then $P(2,3;2,4)=\left(\begin{smallmatrix} f&h \\ j&l \end{smallmatrix}\right) $.

Given two matrices $P$ and $Q$, we say $P$ contains $Q$, whenever $Q$ is equal to a submatrix of $P$. Otherwise say $P$ avoids $Q$.
For example, $\left(\begin{smallmatrix} 1&0&1 \\ 0&0&1 \\ 1&0&0 \end{smallmatrix}\right)$ contains $\left(\begin{smallmatrix} 0&1 \\ 1&0 \end{smallmatrix}\right)$
but avoids $\left(\begin{smallmatrix} 1&0 \\ 0&1 \end{smallmatrix}\right)$.
For a matrix $P$ and a set $\alpha$ of matrices, we say that $P$ avoids the type set $\alpha$ if $P$ avoids all the matrices in $\alpha$. If it causes no confusion we will simply say that $P$ avoids $\alpha$.

Given a set $\alpha$ of matrices, let $\phi(k,n;\alpha)$ be the number of $\kn$ matrices avoiding $\alpha$. From the definition of $M(k,n)$, for any set $\alpha$, we have $\phi(0,0;\alpha)=1$ and $\phi(k,0;\alpha)=\phi(0,n;\alpha)=0$ for positive integers $k$ and $n$. Let $\Phi(x,y;\alpha)$ be the bivariate exponential generating function for $\phi(k,n;\alpha)$, i.e.,
$$
\Phi(x,y;\alpha):= \sum_{n \ge 0}\sum_{k \ge 0} \phi(k,n;\alpha)\,\frac{x^k}{k!}\frac{y^n}{n!}= 1+ \sum_{n\ge 1}\sum_{k \ge 1} \phi(k,n;\alpha)\,\frac{x^k}{k!}\frac{y^n}{n!}\,.
$$
Let $\Phi(z;\alpha)$ be the exponential generating function for $\phi(n,n;\alpha)$, i.e.,
$$
\Phi(z;\alpha):= \sum_{n \ge 0} \phi(n,n;\alpha)\,\frac{z^n}{n!}\,.
$$
Given $f, g \in {\mathbb C}[[x,y]]$, we denote $f \stackrel{2}= g$ if the coefficients of $x^k y^n$ in $f$ and $g$ are the same, for each $k,n \ge 2$.

\subsection{$I$-avoiding matrices (Lonesum matrices)} \label{sub:ls}
This is related to the lonesum matrices. A lonesum matrix is a $(0,1)$-matrix determined uniquely by its column-sum and row-sum vectors.
For example, $\left(\begin{smallmatrix} 1&0&1 \\ 0&0&1 \\ 1&0&1 \end{smallmatrix}\right) $
is a lonesum matrix since it is a unique matrix determined by the column-sum vector $(2,0,3)$ and the row-sum vector $(2,1,2)^{\rm t}$.
However
$\left(\begin{smallmatrix} 1&0&1 \\ 0&0&1 \\ 1&0&0 \end{smallmatrix}\right) $
is not, since $\left(\begin{smallmatrix} 1&0&1 \\ 1&0&0 \\ 0&0&1 \end{smallmatrix}\right) $
has the same column-sum vector $(2,0,2)$ and row-sum vector $(2,1,1)^{\rm t}$.

%For a given column-sum and row-sum vector, {\em Ryser class} is defined as the set of all matrices with the same column-sum vector and row-sum
%vector as given. Hence if $M$ is a $(0,1)$-lonesum matrix its Ryser class is the singleton set $\{M\}$.

\begin{thm}[Brewbaker~\cite{Br}] \label{Br}
A matrix is a lonesum matrix if and only if it avoids $I$.
\end{thm}

Theorem~\ref{Br} implies that $\phi(k,n; I)$ is equal to the number of $\kn$ lonesum matrices.

\begin{dfn}
Bernoulli number $B_n$ is defined as following:
%$$
%B_0 = 1, \qquad \sum_{i=0}^{m} \binom{m+1}{i}\,B_i =0.
%$$
$$
\sum_{n\geq 0} B_n \frac{x^n}{n!} = \frac{x\,e^x}{e^x -1}.
$$
\end{dfn}
%The exponential generating function for the Bernoulli number is
%$$
%\sum_{n\geq 0} B_n \frac{x^n}{n!} = \frac{x\,e^x}{e^x -1}.
%$$
Note that $B_n$ can be written explicitly as
$$
B_n = \sum_{m=0}^{n} (-1)^{m+n} \,\frac{m! \,S(n,m)}{m+1},
$$
where $S(n,m)$ is the Stirling number of the second kind. The poly-Bernoulli number, introduced first by Kaneko~\cite{Ka}, is defined as
%$$
%B_n^{(k)} = \sum_{m=0}^{n} (-1)^{m+n}\, \frac{m!\, S(n,m)}{(m+1)^k},
%$$
%and its exponential generating function is
$$
\sum_{n\geq0} B_n^{(k)} \frac{x^n}{n!} = \frac{{\rm Li}_k(1-e^{-x})}{1-e^{-x}},
$$
where the polylogarithm ${\rm Li}_k(x)$ is the function ${\rm Li}_k(x):= \sum_{m\geq 1} \frac{x^m}{m^k}$.
Bernoulli numbers are nothing but poly-Bernoulli numbers with $k=1$.
Sanchez-Peregrino~\cite{SP} proved that $B_n^{(-k)}$ has the following simple expression:
$$
B_n^{(-k)} = \sum_{m=0}^{\min(k,n)} (m!)^2 \,S(n+1,m+1)\, S(k+1, m+1)\,.
$$

Brewbaker~\cite{Br} and Kim et. al.~\cite{KKL} proved that the number of $\kn$ lonesum matrices is the poly-Bernoulli number $B_n^{(-k)}$, which yields the following result.
\begin{prop}[Brewbaker~\cite{Br}; Kim, Krotov, Lee~\cite{KKL}]\label{Br-Kim}
The number of $k\by n$ matrices avoiding $I$ is equal to $B_n^{(-k)}$, i.e.,
\begin{equation} \label{eq:phi-A1}
\phi(k,n;I)= \sum_{m=0}^{\min(k,n)} (m!)^2 \,S(n+1,m+1)\, S(k+1, m+1)\,.
\end{equation}
%In particular, for the square matrices of size $n$, we have
%$$
%\phi(n,n;I)= B_n^{(-n)} = \sum_{m=0}^{n} (m!)^2 \,S(n+1,m+1)^2\,.
%$$
\end{prop}

The generating function $\Phi(x,y;I)$, given by Kaneko~\cite{Ka}, is
\begin{equation} \label{eq:Phi-A1}
%1+\sum_{n\geq 1}\sum_{k \ge 1} \,\phi(k,n;I)\,\frac{x^k}{k!}\frac{y^n}{n!}
 \Phi(x,y;I)=e^{x+y}\sum_{m \geq 0}\left[ (e^x -1)(e^y-1)\right]^m = \frac{e^{x+y}}{e^{x}+e^{y}-e^{x+y}}.
\end{equation}
Also, $\Phi(z;I)$ can be easily obtained as follows:
\begin{eqnarray}
\Phi(z;I)=\sum_{n\geq 0}\,\phi(n,n;I)\,\frac{z^n}{n!} %\notag\\
&=&\sum_{n\geq 0}\sum_{m \ge 0}(-1)^{m+n} m!S(n,m) (m+1)^n\,\frac{z^n}{n!} \notag\\
&=&\sum_{m\geq 0}(-1)^m m!\sum_{n \ge 0}S(n,m)\, \frac{(-(m+1)z)^n}{n!}\notag\\
&=& \sum_{m \ge 0} (1-e^{-(m+1)z})^m\,. \label{eq:Phiz-A1}
\end{eqnarray}

\section{Main Results}
\subsection{$\Gamma$-avoiding matrices (or $C$)}
By row exchange and column exchange we can change the original matrix into a block matrix as in Figure~\ref{fig:A2}.
Here $\bf [0]$ (resp. $\bf [1]$) stands for a $0$-block (resp.$1$-block) and
$\bf [0^*]$ stands for a $0$-block or an empty block.
Diagonal blocks are $\bf [1]$'s except for the last block $\bf [0^*]$, and the off-diagonal blocks are $\bf [0]$'s.

\begin{figure}
\begin{center}
\includegraphics[scale=.5]{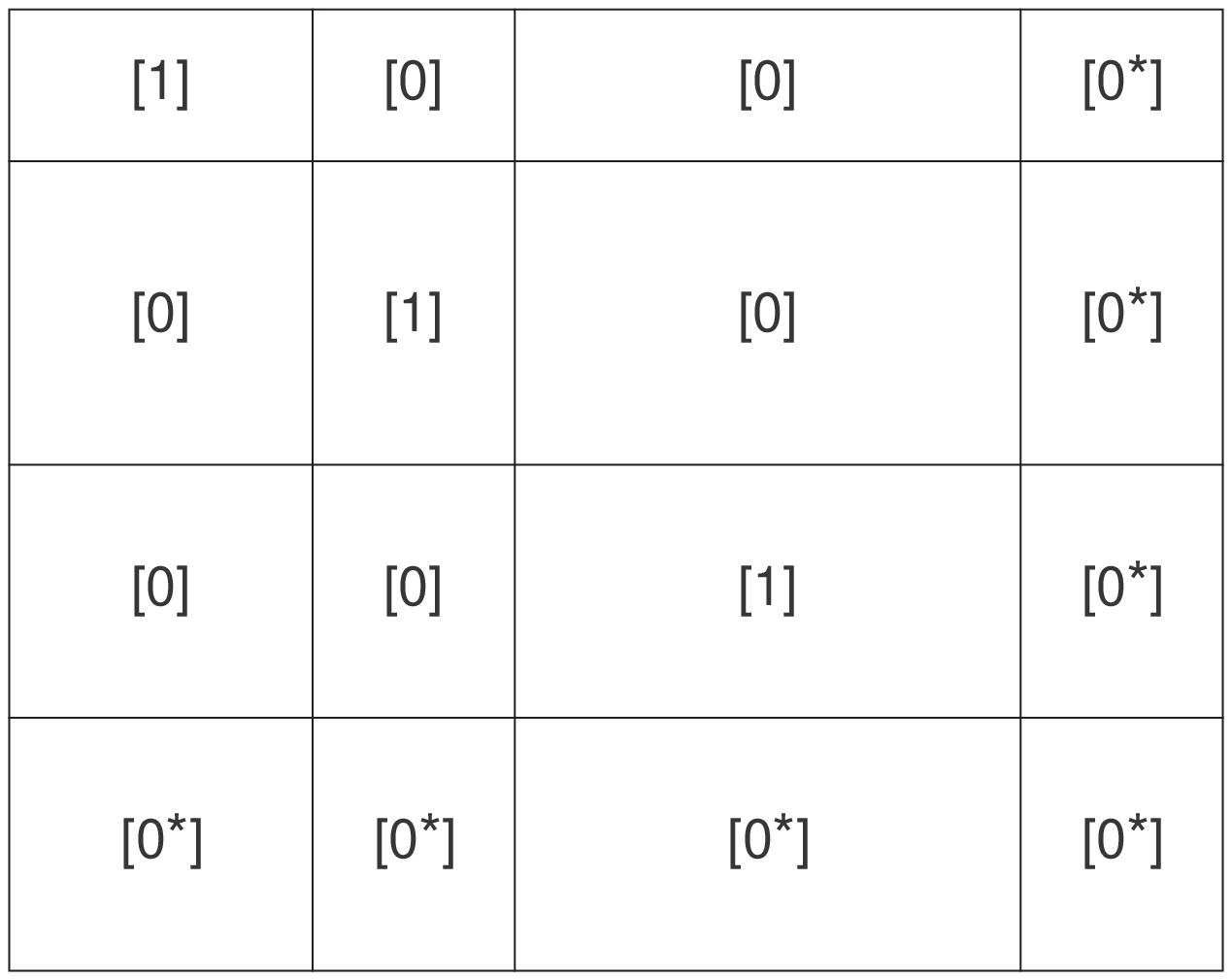}
\caption{A matrix avoiding $\Gamma$ can be changed into a block diagonal matrix.}\label{fig:A2}
\end{center}
\end{figure}

\begin{thm}\label{thm:A2} The number of $\kn$ matrices avoiding $\Gamma$ is given by
\begin{equation}\label{eq:phi-A2}
\phi(k,n;\Gamma)= \sum_{m=0}^{\min(k,n)} m! \,S(n+1,m+1)\, S(k+1, m+1).
\end{equation}
%In particular, for the square matrix of size $n$, we have
%$$
%\phi(n,n;\Gamma)= \sum_{m=0}^{n} m! \,S(n+1,m+1)^2.
%$$
\end{thm}
\begin{proof}
Let $\mu=\{C_1, C_2, \ldots, C_{m+1}\}$ be a set partition of $[n+1]$ into $m+1$ blocks. Here the block $C_l$'s are ordered by the largest element of each block. Thus  $n+1$ is contained in $C_{m+1}$. Likewise,
let $\nu=\{D_1, D_2, \ldots, D_{m+1}\}$ be a set partition of $[k+1]$ into $m+1$ blocks. Choose $\sigma \in S_{m+1}$ with $\sigma(m+1)=m+1$, where $S_{m+1}$ is the set of all permutations of length $m+1$.
Given $(\mu,\nu,\sigma)$ we define a $\kn$ matrix $M=(a_{i,j})$ by
$$
a_{i,j}:=
\begin{cases} 1,& (i,j)\in C_l \times D_{\sigma(l)}~\text{for some $l\in [m]$} \\ 0,&\text{otherwise}
\end{cases} \,.
$$
It is obvious that the matrix $M$ avoids the type $\Gamma$.

Conversely, let $M$ be a $\kn$ matrix avoiding type $\Gamma$. Set $(k+1)\times(n+1)$ matrix ${\widetilde M}$ by augmenting zeros to the last row and column of $M$. By row exchange and column exchange we can change ${\widetilde M}$ into a block diagonal matrix $B$, where each diagonal is $1$-block except for the last diagonal. By tracing the position of columns (resp. rows) in ${\widetilde M}$, $B$ gives a set partition of $[n+1]$ (resp. $[k+1]$). Let $\{C_1, C_2, \ldots, C_{m+1}\}$ (resp. $\{D_1, D_2, \ldots, D_{m+1}\}$) be the set partition of $[n+1]$ (resp. $[k+1]$). Note that
the block $C_i$'s and $D_i$'s are ordered by the largest element of each block. Let $\sigma$ be a permutation on $[m]$ defined by $\sigma(i)=j$ if $C_i$ and $D_j$ form a $1$-block in $B$.

The number of set partitions $\pi$ of $[n+1]$ is $S(n+1,m+1)$, and the number of set partitions $\pi'$ of $[k+1]$ is $S(n+1,k+1)$. The cardinality of the set of $\sigma$'s is the cardinality of $S_{m}$, i.e., $m!$.
Since the number of blocks $m+1$ runs through $1$ to $\min(k,n)+1$, the sum of $S(k+1,m+1)\,S(n+1,m+1)
\,m!$ gives the required formula.
\end{proof}

\begin{eg}
Let $\mu=4/135/26/7$ be a set partition of $[7]$ and $\nu=25/6/378/149$ of $[9]$ into $4$ blocks. Let $\sigma=3124$ be a permutation
in $S_4$ such that $\sigma(4)=4$.
From $(\mu,\nu,\sigma)$ we can construct the $6\times8$ matrix $M$ which avoids type $\Gamma$ as in Figure~\ref{fig:A2eg}.
\begin{figure}
\begin{center}
%--- FIGURE A2 here ---
\includegraphics[scale=.79]{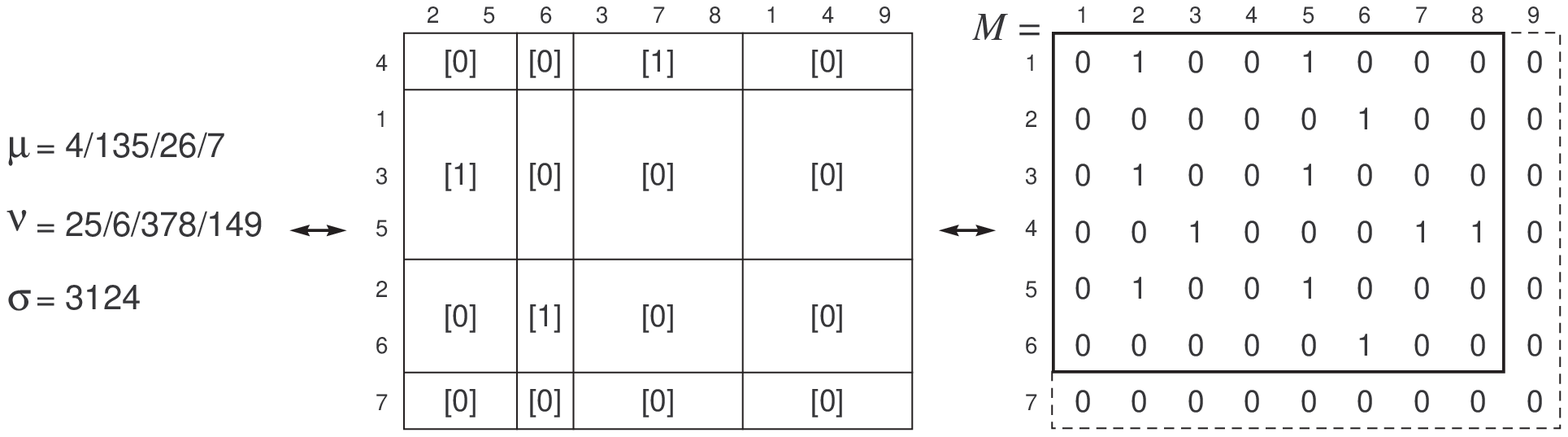}
\caption{A matrix avoiding $\Gamma$ corresponds to two set partitions with a permutation.}\label{fig:A2eg}
\end{center}
\end{figure}
\end{eg}

To find the generating function for $\phi(k,n;\Gamma)$ the following formula (see~\cite{GKP}) is helpful.
\begin{equation} \label{eq:GKP}
\sum_{n \geq 0} S(n+1,m+1) \,\frac{x^n}{n!} = e^x\,\frac{(e^x - 1)^m}{m!}.
\end{equation}

From Theorem~\ref{thm:A2} and \eqref{eq:GKP}, we can express $\Phi(x,y;\Gamma)$ as follows:
\begin{eqnarray}
\Phi(x,y;\Gamma)
%=\sum_{n,k\geq 0}\,\phi(k,n;\Gamma)\,\frac{x^k}{k!}\frac{y^n}{n!} %\notag\\
&=&\sum_{n,k\geq 0}\,\sum_{m\ge 0} m! \,S(n+1,m+1)\, S(k+1, m+1)\,\frac{x^k}{k!}\frac{y^n}{n!}\notag\\
&=&\sum_{m\ge 0}\frac{1}{m!}\,\sum_{k\geq 0} \,m!\,S(k+1,m+1) \frac{x^k}{k!} \,\sum_{n\ge 0} m!\,S(n+1, m+1)\,\frac{y^n}{n!}\notag\\
%&=&\sum_{m\geq 0} \,e^x (e^x -1)^m \sum_{n\ge 0}\,S(n+1,m+1) \frac{y^n}{n!} \notag\\
&=&\sum_{m\geq 0} \,\frac{1}{m!}\,e^x (e^x -1)^m  \,\,e^y(e^y-1)^m \notag\\
&=& \exp[(e^x-1)(e^y-1)+x+y]\,. \label{eq:Phi-A2}
\end{eqnarray}
\begin{rmk}
It seems to be difficult to find a simple expression of $\Phi(z;\Gamma)$.
The sequence $\phi(n,n;\Gamma)$ is not listed in the OEIS~\cite{OEIS}. The first few terms of  $\phi(n,n;\Gamma)$ ($0 \le n \le 9$) are as follows:
$$
1, 2, 12, 128, 2100, 48032, 1444212, 54763088, 2540607060, 140893490432, \ldots
$$
\end{rmk}

%\subsection{Type $\Gamma\cup C$}
\subsection{$\{\Gamma, C\}$-avoiding matrices}
Given the equivalence relation $\sim$ on $M(2,2)$, which is defined in Section~1, if we define the new equivalent relation $P\sim ' Q$ by $P\sim Q$ or $P\sim\sMTI - Q$, then $[\Gamma] \cup [C]$ becomes a single equivalent class. Clearly $\phi(k,n;\{\Gamma,C \})$ is the number of $k\times n$ $(0,1)$-matrix which does not have a submatrix in $[\Gamma] \cup [C]$. From now on we simply write $\phi(k,n;\Gamma,C)$, instead of $\phi(k,n;\{\Gamma,C \})$.
%As mentioned in Section 1, if we add a new relation -- exchanging $0$ and $1$ -- on matrices, then $\Gamma \cup C$ becomes a new equivalent class.
The reduced form of a matrix $M$ avoiding~$\{\Gamma , C\}$ is very simple as in Figure~\ref{fig:A23}.
\begin{figure}
\begin{center}
\includegraphics[scale=.7]{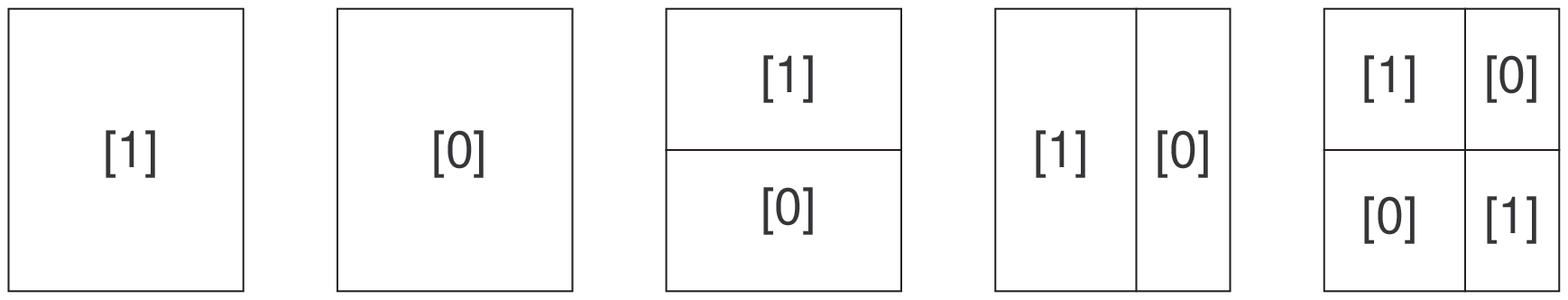}
\caption{Possible reduced forms of matrices avoiding $\{\Gamma , C\}$}\label{fig:A23}
\end{center}
\end{figure}
%Here $\bf [1^*]$ stands for a $1$-block or an empty block.
In this case if the first row and the first column of $M$ are determined then the rest of the entries of $M$ are determined uniquely.
Hence the number $\phi(k,n;\Gamma , C)$ of such matrices is
\begin{equation} \label{eq:phi-A23}
\phi(k,n;\Gamma , C)=2^{k+n-1}\qquad (k,n \ge 1),
\end{equation}
and its exponential generating function is
\begin{equation}\label{eq:Phi-A23}
\Phi(x,y;\Gamma , C)
= 1+\frac12\,(e^{2x}-1)(e^{2y}-1)\,.
\end{equation}
Clearly, $\phi(n,n;\Gamma , C)=2^{2n-1}$ for $n \geq 1$. Thus its exponential generating function is
\begin{equation}\label{eq:Phiz-A23}
\Phi(z;\Gamma , C)
= \frac12\,(e^{4z}+1)\,.
\end{equation}

\subsection{$T$-avoiding matrices (or $L$)}
Given a $(0,1)$-matrix, $1$-column (resp. $0$-column) is a column in which all entries consist of $1$'s (resp. $0$'s). We denote a $1$-column (resp. $0$-column) by $\bf 1$ (resp. $\bf 0$). A mixed column is a column which is neither $\bf 0$ nor $\bf 1$.
For $k=0$, we have $\phi(0,n;T)=\delta_{0,n}$. In case $k \ge 1$, i.e., there being at least one row, we can enumerate as follows:
\begin{itemize}
\item case 1: there are no mixed columns.
 Then each column should be $\bf 0$ or $\bf 1$. The number of such $k\by n$ matrices is $2^n$.
\item case 2: there is one mixed column.
 In this case each column should be $\bf 0$ or $\bf 1$ except for one mixed column.
 The number of $k\by n$ matrices of this case is $2^{n-1}\,n\,(2^k -2)$.
\item case 3: there are two mixed columns.
 As in case 2, each column should be $\bf 0$ or $\bf 1$ except for two mixed columns, say, $v_1$ and $v_2$.
 The number of $k\by n$ matrices of this case is the sum of the following three subcases:
 %\begin{enumerate}
\subitem - $v_1 + v_2= {\bf 1}$\,: ~$2^{n-2}\binom{n}{2}\,2!\,S(k,2)$
\subitem - $v_1 + v_2$ has an entry $0$\,: ~$2^{n-2}\binom{n}{2}\,3!\,S(k,3)$
\subitem - $v_1 + v_2$ has an entry $2$\,: ~$2^{n-2}\binom{n}{2}\,3!\,S(k,3)$
% \end{enumerate}
 %$$
% 2^{n-2}\binom{n}{2}\,\left(\, 2\cdot3!\,S(k,3)+2!\,S(k,2)\,\right).
% $$
\item case 4: there are $m$ ($m \geq 3$) mixed columns $v_1,\ldots,v_m$.
The number of $k\by n$ matrices of this case is the sum of the following four subcases:
 %\begin{enumerate}
\subitem - $v_1+\cdots+v_m={\bf 1}$\,: ~$2^{n-m}\binom{n}{m}\,m!\,S(k,m)$
\subitem - $v_1+\cdots+v_m = (m-1){\bf 1}$\,: ~$2^{n-m}\binom{n}{m}\,m!\,S(k,m)$
\subitem - $v_1+\cdots+v_m$ has an entry $0$\,: ~$2^{n-m}\binom{n}{m}\,(m+1)!\,S(k,m+1)$
\subitem - $v_1+\cdots+v_m$ has an entry $m$\,: ~$2^{n-m}\binom{n}{m}\,(m+1)!\,S(k,m+1)$
% \end{enumerate}
%
% $$
% 2^{n-m}\binom{n}{m}\left(\, 2\cdot(m+1)!\,S(k,m+1)+2\cdot m!\,S(k,m)\,\right).
% $$
\end{itemize}
Adding up all numbers in the previous cases yields the following theorem.
\begin{thm} \label{thm:phi-A4}
For $k,\,n \geq 1$ the number of $\kn$ matrices avoiding $T$ is given by
\begin{equation} \label{eq:phi-A4}
\phi(k,n;T)=2\sum_{l \ge 1}\,\binom{n}{l-1}\,l^k + (n^2 -n -4)\,2^{n-2} - n(n+3)\,2^{n+k-3}.
\end{equation}
\end{thm}
\begin{proof}
%for $n \geq 1$\ and $k \geq 1$:
\begin{eqnarray*}
\phi(k,n;T)
&=&2^n + 2^{n-1}\binom{n}{1}\,(2^{k}-2)+2^{n-2}\binom{n}{2}\,\left(\, 2!\,S(k,2)+3!\,2\,S(k,3)\,\right) \notag \\
& &+ \,\sum_{m=3}^{n}\, 2^{n-m+1}\binom{n}{m}\left(\, m!\,S(k,m)+(m+1)!\,S(k,m+1)\,\right) \notag\\
%&=&\sum_{m=0}^{n}\,\binom{n}{m}2^{n-m}\left(\, 2\,S(k,m+1)+2\,S(k,m)\,\right)  \\
%& &-\, 2^n - \binom{n}{1}2^{n-1}(S(k,2)+2\,S(k,1))-\binom{n}{2}2^{n-2}S(k,2) \\
&=& 2\sum_{m=0}^{n}\,2^{n-m}\binom{n}{m}\,m! S(k+1,m+1)%\\
%& &
+\, (n^2 -n -4)\,2^{n-2} - n(n+3)\,2^{n+k-3}\notag\\
&=& 2\sum_{l \ge 1}\,\binom{n}{l-1}\,l^k + (n^2 -n -4)\,2^{n-2} - n(n+3)\,2^{n+k-3}. %\label{eq:phi-A4-pf}
\end{eqnarray*}
\end{proof}
Note that in the proof of Theorem~\ref{thm:phi-A4} we use the identity
$$
\sum_{m \ge 0} \binom{n}{m}\, m!\, S(k,m) \,2^{n-m} = \sum_{l\ge 0} \binom{n}{l} \,l^k\,,
$$
where both sides count the number of functions $f$ from $[k]$ to $[n]$ such that each element of~$[n]\setminus f([k])$ has two colors.
%\end{proof}

The generating function $\Phi(x,y;T)$ is given by
\begin{eqnarray}
%\sum_{n\geq 0}\sum_{k\geq 0} \,\phi(k,n;T)\,\frac{x^k}{k!}\frac{y^n}{n!}
\Phi(x,y;T)&\!=\!&1+ \sum_{n\geq 1}\sum_{k\geq 1} \,2\sum_{l \ge 1}\,\binom{n}{l-1}\,l^k\,\frac{x^k}{k!}\frac{y^n}{n!}\notag\\
&\!\!&+ \sum_{n\geq 1}\sum_{k\geq 1} \,(n^2 -n -4)\,2^{n-2}\,\frac{x^k}{k!}\frac{y^n}{n!}- \sum_{n\geq 1}\sum_{k\geq 1} \,n(n+3)\,2^{n+k-3}\,\frac{x^k}{k!}\frac{y^n}{n!}\notag\\
&\!=\!&1+ \left(2e^x (e^{y(e^x +1)}-1)-2e^{2y}+2 \right) \notag\\
&\!\!&+ (e^x-1)\left((y^2-1)e^{2y}+1\right) - \frac12 y(y+2)e^{2y}(e^{2x} -1)  \notag\\
&\!=\!& 2e^{y(e^x\!+1)+x}\!-\!\frac{y^2\!+2y}{2}e^{2x+2y}\!
+\!(y^2\!-\!1)e^{x+2y}\!-e^x\!-\frac{y^2\!-\!2y\!+\!2}{2}e^{2y}\!+\!2\,. \label{eq:Phi-A4}
\end{eqnarray}
Note that if we use the symbol ``$\stackrel{2}=$" introduced in Section~\ref{sub:dn}, then
$$\Phi(x,y;T)\stackrel{2}= 2e^{y(e^x\!+1)+x}\!-\!\frac{y^2\!+2y}{2}e^{2x+2y}\!
+\!(y^2\!-\!1)e^{x+2y}\,.$$

For the $n\by n$ matrices we have
$$
\phi(n,n;T) =  2\sum_{l \ge 1}\,\binom{n}{l-1}\,l^n + (n^2 -n -4)\,2^{n-2} - n(n+3)\,2^{2n-3}.
$$
Thus the generating function $\Phi(n,n; T)$ is given by
\begin{eqnarray}
\sum_{n\geq 0}\,\phi(n,n;T)\,\frac{z^n}{n!}
&=&2\sum_{n\geq 0}\,\sum_{l \ge 1}\,\binom{n+1}{l}\,l^{n+1}\,\frac{z^n}{(n+1)!} \notag\\
& &~+ \sum_{n\geq 0}\,\frac{n^2 -n -4}{4}\,\frac{(2z)^n}{n!}
- \frac{n(n+3)}{8}\,\frac{(4z)^n}{n!} \notag\\
&=& \frac2z \sum_{l \ge 1} \frac{(lz)^{l}}{l!} \sum_{n\geq l-1} \frac{(lz)^{n-l+1}}{(n-l+1)!} +(z^2 -1)e^{2z}-2z(z+1)e^{4z}\notag\\
&=& \frac2z \sum_{l \ge 1} \frac{l^{l}}{l!} (ze^z)^l +(z^2 -1)e^{2z}-2z(z+1)e^{4z}\notag\\
&=& \frac2z \,\left( ze^z\,W'(-ze^{z})\right) +(z^2 -1)e^{2z}-2z(z+1)e^{4z}\notag\\
&=& \frac{-2\,W(-ze^{z})}{z+z\,W(-ze^z)}+(z^2 -1)e^{2z}-2z(z+1)e^{4z}, \label{eq:Phiz-A4}
\end{eqnarray}
where
$$W(x):=\sum_{n \ge 1} (-n)^{n-1} \frac{x^n}{n!}$$
is the Lambert $W$-function which is the inverse function of $f(W)=We^W$.
See~\cite{CGHJK} for extensive study about the Lambert $W$-function.

\begin{rmk}
%\begin{enumerate}
%\item It is hard (but not impossible) to find the generating function for $\phi(n,n;T)$.
The sequence $\phi(n,n;T)$ is not listed in the OEIS \cite{OEIS}. The first few terms of~$\phi(n,n;T)$ ($0 \le n \le 9$) are as follows:
$$
1, 2, 14, 200, 3536, 67472, 1423168, 34048352, 927156224, 28490354432, \ldots
$$
%\end{enumerate}
\end{rmk}

\subsection{$\{T,L\}$-avoiding matrices}
Given the equivalence relation $\sim$ on $M(2,2)$, which is defined in Section~1, if we define the new equivalent relation $P\sim '' Q$ by $P\sim Q$ or $P\sim Q^{\rm t}$, then $[T] \cup [L]$ becomes a single equivalent class. Clearly $\phi(k,n;\{T,L \})$ is the number of $k\times n$ $(0,1)$-matrix which does not have any submatrix in $[T] \cup [L]$.
%If we add a new relation -- transpose -- on matrices, then $T \cup L$ becomes a new equivalent class.
By the symmetry of $\{T, L\}$, we have
$$\phi(k,n;T , L)=\phi(n,k;T , L).$$
So it is enough to consider the case
$k \geq n$.
For $k \leq 2$ or $n\leq 1$, we have
\begin{gather*}
\phi(0,n;T , L)=\delta_{0,n}, \quad \phi(1,n;T , L)=2^n ,\\
\phi(k,0;T , L)=\delta_{k,0}, \quad \phi(k,1;T , L)=2^k , \\
\phi(2,2;T , L)=12.
\end{gather*}

Given a $(0,1)$-vector $v$ with a length of at least $3$, $v$ is called $1$-dominant (resp. $0$-dominant) if all entries of $v$ are $1$'s (resp. $0$'s) except one entry.

\begin{thm} For $k \geq 3$ and $n \geq 2$, the number of $\kn$ matrices avoiding $\{T, L\}$ is equal to twice the number of rook positions in the $k\times n$ chessboard. In other words,
\begin{equation}
\phi(k,n;T , L) =2\sum_{m=0}^{\min(k,n)}\,\binom{k}{m}\binom{n}{m}m!\,. \label{eq:phi-A45}
\end{equation}
\end{thm}
\begin{proof}
Suppose $M$ is a $\kn$ matrix avoiding $\{T, L\}$. It is easy to show each of the following steps:
\begin{itemize}
\item[(i)] If $M$ has a mixed column $v$, then $v$ should be either $0$-dominant or $1$-dominant.
\item[(ii)] Assume that $v$ is $0$-dominant. This implies that other mixed columns(if any) in $M$ should be $0$-dominant.
\item[(iii)] Any non-mixed column in $M$ should be a $0$-column.
\item[(iv)] The location of $1$'s in $M$ corresponds to a rook position in the $k\times n$ chessboard.
\end{itemize}
If we assume $v$ is $1$-dominant in (ii) then the locations of $0$'s again corresponds to a rook position. The summand of RHS in \eqref{eq:phi-A45} is the number of rook positions in the $k\times n$ chessboard with $m$ rooks. This completes the proof.
\end{proof}

The generating function $\Phi(x,y;T , L)$ is given by
\begin{align}
\Phi(x,y&;T , L)
\,=\, 2e^{xy+x+y} -\frac{(xy)^2}{2} \notag\\
&-2xy+3 - 2e^x - 2e^y + x(e^y - 2y -1)(e^y -1) + y(e^x -2x -1)(e^x -1). \label{eq:Phi-A45}
\end{align}
Here the crucial part of the equation~\eqref{eq:Phi-A45} can be obtained as follows:
\begin{eqnarray*}
\sum_{k,n \ge 0} \left( \sum_{m \ge 0}\,\binom{k}{m}\binom{n}{m} m!\right) \frac{x^k}{k!}\frac{y^n}{n!}
&=& \sum_{m \ge 0} m! \left( \sum_{k \ge 0}\binom{k}{m} \frac{x^k}{k!} \right)
\left( \sum_{n \ge 0}\binom{n}{m} \frac{y^n}{n!} \right)\\
&=& \sum_{m \ge 0} m! \left(\frac{x^m}{m!} \,e^x \right) \left(\frac{y^m}{m!} \,e^y \right)
~=~ \exp({xy+x+y}).
\end{eqnarray*}
Note that if we use the symbol ``$\stackrel{2}=$" introduced in Section~\ref{sub:dn}, then
$$\Phi(x,y;T,L)\stackrel{2}= 2e^{xy+x+y} -\frac{(xy)^2}{2}.$$
%i.e., the summands in the second line of \eqref{eq:Phi-A45} contribute the coefficient of $x^k y^n$ where $k$ or $n$ are less than $2$.

For the $n\by n$ matrices we have %$\phi(n,n)$
\begin{gather*}
\phi(0,0;T , L)=1, \quad \phi(1,1;T , L)=2, \quad \phi(2,2;T , L)=12, \text{~~and}\\
\phi(n,n;T , L)=2\,\sum_{m=0}^{n}\,\binom{n}{m}^2 m!\,. \quad (n\geq 3)
\end{gather*}
Thus the generating function $\Phi(z;T , L)$ is given by
\begin{equation}\label{eq:Phiz-A45}
\Phi(z;T , L)
= {2\,e^{\frac{z}{1-z}} \over 1-z}\,-1-2z-z^2.
\end{equation}
Note that we use the equation
$$
\sum_{n \ge 0} \left( \sum_{m=0}^{n}\,\binom{n}{m}^2 m!\right) \frac{z^n}{n!} = \frac{e^{\frac{z}{1-z}}}{1-z},
$$
which appears in \cite[pp. 597--598]{FS}.

\subsection{$\{J,O\}$-avoiding matrices}\label{subJO}
Recall the equivalence relation $\sim'$ defined in subsection~3.2. With this relation, $\{J, O\}$ becomes a single equivalent class.
%Clearly $\phi(k,n;\{T,L \})$ is the number of $k\times n$ $(0,1)$-matrix does not have a submatrix in $[T] \cup [L]$.
%If we add a relation -- exchanging $0$ and $1$ -- on $(0,1)$-matrices, then $J \cup O$ becomes a new equivalent class.
Due to the symmetry of $\{J, O\}$  it is obvious that
$$\phi(k,n;J , O)=\phi(n,k;J , O).$$
The $k$-color bipartite Ramsey number $br(G;k)$ of a bipartite graph $G$ is the minimum integer $n$ such that
in any $k$-coloring of the edges of $K_{n,n}$ there is a monochromatic subgraph isomorphic to $G$.
Beineke and Schwenk \cite{BS} had shown that $br(K_{2,2};2)=5$.
%, and Exoo \cite{Exo} showed that $br(G;3)=11$.
From this we can see that
$$\phi(k,n;J , O)=0 \quad (k,n \geq 5)\,.$$

For $k=1$ and $2$, we have
\begin{equation*}
\phi(1,n;J ,O)=2^n, \qquad %\text{and} \quad
\phi(2,n;J , O)=(n^2+3n+4)2^{n-2}.
\end{equation*}
Note that the sequence $(n^2+3n+4)\,2^{n-2}$ appears in \cite[A007466]{OEIS} and its exponential generating function is
$(1+x)^2 e^{2x}$.

For $k \geq 3$, we have
\begin{eqnarray*}
\phi(3,n;J , O)=\phi(4,n;J , O)=0 &\quad&\mbox{for $n\geq 7$}, \\
\phi(5,n;J , O)=\phi(6,n;J , O)=0 &\quad&\mbox{for $n\geq 5$}, \\
\phi(k,n;J , O)=0 &\quad&\mbox{for $k\geq 7$ and $n\geq 3$}.
\end{eqnarray*}

For exceptional cases, due to the symmetry of $\{J , O\}$, it is enough to consider the followings:
\begin{gather*}
\phi(3,3;J, O)=156,\quad \phi(3,4;J, O)=408,\quad \phi(4,4;J, O)=840, \\
\phi(3,5;J, O)=\phi(3,6;J, O)=\phi(4,5,J, O)=\phi(4,6;J, O)=720\,.
\end{gather*}
The sequence $\phi(k,n;J, O)$ is listed in Table~\ref{tab:A67}. Note that Kitaev et al. have already calculated $\phi(k,n;J, O)$ in \cite[Proposition 5]{KMV}, but the numbers of $\phi(3,3;J, O)$ and $\phi(4,4;J, O)$ are different with ours.
\begin{table}[t]\centering
\begin{tabular}{|c||c|c|c|c|c|c|c|c|}
  \hline
  $k\diagdown n$ & 1 & 2 & 3 & 4 & 5 & 6 & 7 & $\cdots$ \\
  \hline\hline
  1 & 2 & 4 & 8 & 16 & 32 & 64 & 128 & $\cdots$ \\
  \hline
  2 & 4 & 14 & 44 & 128 & 352 & 928 & 2368 & $\cdots$ \\
  \hline
  3 & 8 & 44 & 156 & 408 & 720 & 720 & 0 & 0 \\
  \hline
  4 & 16 & 128 & 408 & 840 & 720 & 720 & 0 & 0 \\
  \hline
  5 & 32 & 352 & 720 & 720 & 0 & 0 & 0 & 0 \\
  \hline
  6 & 64 & 928 & 720 & 720 & 0 & 0 & 0 & 0 \\
  \hline
  7 & 128 & 2368 & 0 & 0 & 0 & 0 & 0 & 0 \\
  \hline
  \vdots & \vdots & \vdots & 0 & 0 & 0 & 0 & 0 & 0 \\
  \hline
\end{tabular}
\caption{The sequence $\phi(k,n;J , O)$}\label{tab:A67}
\end{table}

The generating function~$\Phi(x,y;J , O)$ is given by
\begin{align}
\Phi(x,y&;J , O)
\,=\, 1+ x\,e^{2y}+y\,e^{2x}+x^2(1+y)^2e^{2y}+y^2(1+x)^2 e^{2x}\notag\\
&-\left(x+y+\frac{x^2}{2!}+\frac{y^2}{2!}+2xy+2x^2 y + 2xy^2 +14\frac{x^2y^2}{2!2!}\right)
+156\frac{x^3y^3}{3!3!}+840\frac{x^4y^4}{4!4!} \notag\\
&+720\left(\frac{x^3y^5}{3!5!}+\frac{x^5y^3}{5!3!}+\frac{x^3y^6}{3!6!}+\frac{x^6y^3}{6!3!}
+\frac{x^4y^5}{4!5!}+\frac{x^5y^4}{5!4!}+\frac{x^4y^6}{4!6!}+\frac{x^6y^4}{6!4!} \right)\,. \label{eq:Phi-A67}
\end{align}
%Note that if we use the symbol ``$\stackrel{2}=$" introduced in Section~\ref{sub:dn}, then
%$$\Phi(x,y;J,O)\stackrel{2}= 14\frac{x^2y^2}{2!2!}
%+156\frac{x^3y^3}{3!3!}+840\frac{x^4y^4}{4!4!} %\notag\\
%+720\left(\frac{x^3y^5}{3!5!}+\frac{x^5y^3}{5!3!}+\frac{x^3y^6}{3!6!}+\frac{x^6y^3}{6!3!}
%+\frac{x^4y^5}{4!5!}+\frac{x^5y^4}{5!4!}+\frac{x^4y^6}{4!6!}+\frac{x^6y^4}{6!4!} \right).$$
In particular, the generating function~$\Phi(z;J , O)$ is given by
\begin{equation}
\Phi(z;J , O) = 1+2z + 7 z^2 + 26 z^3 + 35 z^4\,.  \label{eq:Phiz-A67}
\end{equation}

\section{Concluding remarks}
Table~\ref{tab:A-all} summarizes our results (except the $I$-avoiding case).
Due to the amount of difficulty, we are not able to enumerate the number $\phi(k,n;J)$, hence $\phi(k,n;O)$.
Note that $\phi(k,n;J)$ is equal to the following:
\begin{itemize}
\item[(a)] The number of labeled $(k,n)$-bipartite graphs with girth of at least $6$, i.e., the number of $C_4$-free labeled
%or $K_{2,2}$   (a cycle of length $4$)%-free
$(k,n)$-bipartite graphs, where $C_4$ is a cycle of length $4$.
\item[(b)] The cardinality of the set %:%of certain ordered $k$-tuples of sets, as follows:
$
\{(B_1, B_2,\ldots, B_k)\,:\,B_i \subseteq [n]~\forall\, i,~~|B_i \cap B_j| \leq 1~\forall\, i\ne j \}
$.
\end{itemize}
For the other subsets $\alpha$ of $\mS$ which is not listed in Table~\ref{tab:A-all}, we have calculated $\phi(k,n;\alpha)$ in \cite{JS2}. Note that if the size of $\alpha$ increases then enumeration of $M(\alpha)$ becomes easier.

For further research, we suggest the following problems.
\begin{enumerate}
%\item Enumeration of sets of $\kn$ $(0,1)$-matrices avoiding each individual matrix instead of each equivalent class by row/column exchange. For example, $\phi(k,n;\{\sMTad \})$.
\item In addition to $(0,1)$-matrices, one can consider $(0,1,\dots,r)$-matrices with $r \ge 2$.
\item Consideration of the results of adding the line sum condition to each individual case given in the first column of Table~\ref{tab:A-all}.
\end{enumerate}
\begin{table}[t]
\centering
\begin{tabular}{|c||c|c|c|}
  \hline
   $\alpha$ & $\phi(k,n;\alpha)$  & $\Phi(x,y;\alpha) $ & $\Phi(z;\alpha)$  \\
  \hline\hline
  $I$ &  \eqref{eq:phi-A1}&  \eqref{eq:Phi-A1}&  \eqref{eq:Phiz-A1} \\
  \hline
  $\Gamma \,(\text{or}~C)$ &  \eqref{eq:phi-A2}&  \eqref{eq:Phi-A2}&  complicated \\
  \hline
  $\{\Gamma , C\}$ &  \eqref{eq:phi-A23}&  \eqref{eq:Phi-A23}&  \eqref{eq:Phiz-A23}  \\
  \hline
  $T \,(\text{or}~L)$ &  \eqref{eq:phi-A4}&  \eqref{eq:Phi-A4}&  \eqref{eq:Phiz-A4}  \\
  \hline
  $\{T , L\}$ &  \eqref{eq:phi-A45}&  \eqref{eq:Phi-A45}&  \eqref{eq:Phiz-A45}  \\
  \hline
  $J \,(\text{or}~O)$ & unknown & unknown & unknown  \\
  \hline
  $\{J, O\}$ & Table~\ref{tab:A67} & \eqref{eq:Phi-A67} & \eqref{eq:Phiz-A67} \\
  \hline
\end{tabular}
\caption{Formulas and generating functions avoiding $\alpha$.}\label{tab:A-all}
%\caption{Formulas and generating functions according to each avoiding type $\alpha$.}\label{tab:A-all}
\end{table}

\section*{Acknowledgement}
The Authors are thankful to Sergey Kitaev in the University of Strathclyde for his helpful comments and suggestions.

%For further research, we suggest the following problems.
%\begin{enumerate}
%\item Enumeration of sets of $\kn$ $(0,1)$-matrices avoiding each individual matrix instead of each equivalent class by row/column exchange. For example, $\phi(k,n;\{\sMTad \})$.
%\item In addition to $(0,1)$-matrices, one can consider $(0,1)/\cdots/r$-matrices with $r \ge 2$.
%\item Consideration of the results of adding the line sum condition to each individual case given in the first column of Table~\ref{tab:A-all}.
%\end{enumerate}

\end{document}